\documentclass[11pt,a4paper]{article}
\usepackage{amssymb} 
\usepackage{amsmath} 
\usepackage{amsthm} 
\usepackage{hyperref}

\newtheorem{theorem}{Theorem}
\numberwithin{theorem}{section}
\newtheorem{proposition}[theorem]{Proposition}

\newtheorem{conjecture}[theorem]{Conjecture}

\newtheorem{claim}[theorem]{Claim}
\newtheorem{problem}[theorem]{Problem}

\newcommand*{\myproofname}{Proof}
\newenvironment{claimproof}[1][\myproofname]{\begin{proof}[#1]}{\end{proof}}

\newcommand{\eps}{\varepsilon}

\DeclareMathOperator{\odd}{odd}

\usepackage{tikz}
\usetikzlibrary{shapes.arrows}
\usetikzlibrary{calc,positioning}
\usetikzlibrary{decorations.pathreplacing}
\usepackage{graphicx}

\tikzstyle{node}=[circle, draw, fill=blue!50,
                        inner sep=0pt, minimum width=4pt]
\tikzstyle{graynode}=[circle, draw, fill=white!50,
                        inner sep=0pt, minimum width=4pt]
\tikzstyle{point}=[circle, draw, fill=black,
                        inner sep=0pt, minimum width=1pt]
\tikzstyle{graypoint}=[circle, draw, fill=gray!50,
                        inner sep=1pt, minimum width=0pt]
\tikzstyle{root}=[circle, draw, fill=blue!50,
                        inner sep=0pt, minimum width=8pt]                    
\tikzstyle{indset}=[circle, draw, fill=blue!50,
                        inner sep=0pt, minimum width=15pt]          
\tikzstyle{bigindset}=[circle, draw, fill=blue!50,
                        inner sep=0pt, minimum width=30pt]

\title{Tree-like distance colouring for planar graphs of sufficient girth}

\author{
Ross J. Kang
\thanks{Radboud University Nijmegen, Netherlands.
Email: \protect\href{mailto:ross.kang@gmail.com}{\protect\nolinkurl{ross.kang@gmail.com}}.
This author is partially supported by a Vidi grant (639.032.614) of the Netherlands Organisation for Scientific Research (NWO).}
\and
Willem van Loon
\thanks{Radboud University Nijmegen, Netherlands.
Email: \protect\href{mailto:willem.g.vanloon@gmail.com }{\protect\nolinkurl{willem.g.vanloon@gmail.com }}.}
}

\begin{document}
\maketitle

\begin{abstract}
Given a multigraph $G$ and a positive integer $t$, the distance-$t$ chromatic index of $G$ is the least number of colours needed for a colouring of the edges so that every pair of distinct edges connected by a path of fewer than $t$ edges must receive different colours.
Let $\pi'_t(d)$ and $\tau'_t(d)$ be the largest values of this parameter over the class of planar multigraphs and of (simple) trees, respectively, of maximum degree $d$.
We have that $\pi'_t(d)$ is at most and at least a non-trivial constant multiple larger than $\tau'_t(d)$.
(We conjecture $\limsup_{d\to\infty}\pi'_2(d)/\tau'_2(d) =9/4$ in particular.)
We prove for odd $t$ the existence of a quantity $g$ depending only on $t$ such that the distance-$t$ chromatic index of any planar multigraph of maximum degree $d$ and girth at least $g$ is at most $\tau'_t(d)$ if $d$ is sufficiently large.
Such a quantity does not exist for even $t$.
We also show a related, similar phenomenon for distance vertex-colouring.

\smallskip
{\footnotesize
{\em Keywords}: graph colouring, distance colouring, planar graphs, girth.}
\end{abstract}

\section{Introduction}\label{sec:intro}
Two classic theorems set a basis for our work.

The first is a result of Vizing from 1965~\cite{Viz65}: every (simple) planar graph of maximum degree $d$ has chromatic index at most $d$, provided $d\ge 8$. (This statement was extended to $d= 7$ by Sanders and Zhao~\cite{SaZh01} but remains open for $d=6$.)
On the other hand, the Shannon multigraphs, i.e.~triangles with edges of balanced multiplicity, are planar and have chromatic index a non-trivial factor (of $3/2$) greater than $d$.
Since the tree consisting of a single vertex with $d$ neighbours has chromatic index $d$, a rough way to view this is as follows: with respect to edge-colouring, planar multigraphs resemble trees if cycles of length $2$ are forbidden (for large enough maximum degree).

The second is a result of Gr\"otzsch from 1959~\cite{Gro59}: every triangle-free planar graph has chromatic number at most $3$. This is sharp due to the odd cycles. And so, even though there is a reduction by one in the number of colours required when the girth (the smallest cycle length) is at least $4$, there is no fixed set of cycle lengths one could forbid to achieve a bound of $2$, the optimal chromatic number over all trees.

We are curious how the above narrative extends if like-coloured elements must be at some minimum distance, and we address the following question.
\begin{quote}\em
For distance colouring of planar (multi)graphs, when does some finite girth constraint ensure the problem resembles that of trees?
\end{quote}
As noted, a girth constraint of $3$ suffices (and is optimal) for ordinary edge-colouring, while for ordinary vertex-colouring no such constraint is possible.

More formally, we study the following parameters. Fix a positive integer $t$. Let $G=(V,E)$ be a multigraph.
The {\em distance-$t$ chromatic index} ({\em number}) $\chi'_t(G)$ ($\chi_t(G)$) of $G$ is the least number of colours needed for a colouring of the edges (vertices) so that every pair of distinct edges (vertices) connected by a path of fewer than $t$ edges ($t+1$ edges) must receive different colours.
For $t=1$ these correspond to the usual chromatic index $\chi'(G)$ and chromatic number $\chi(G)$ of $G$. Moreover, if $H^t$ represents the graph whose adjacency matrix has $1$'s only where the off-diagonal entries in the $t$-th power of the adjacency matrix of $H$ are nonzero, then $\chi_t(G)=\chi(G^t)$ and $\chi'_t(G) = \chi(L(G)^t)$, where $L(G)$ denotes the line graph of $G$.

These natural strengthened colouring parameters have been studied for almost half a century~\cite{KrKr69}. They have attracted much attention particularly in the cases of $\chi'_2$, also called the {\em strong chromatic index}, cf.~e.g.~\cite{MoRe97}, and $\chi_2$, cf.~e.g.~\cite{Weg77,HHMR08}. That is not only because of their obvious mathematical appeal, but also because of their applicability in various other domains, such as the approximation of sparse Hessian matrices~\cite{McC83}, link and broadcast scheduling in wireless networks~\cite{LlRa92,RaLl93}, algorithmic hardness of approximation~\cite{Lae14}.
The results of the present paper offer basic insights into fundamental characteristics of these important graph parameters.

The low waterline for our study is when $G$ is of highest possible girth, i.e.~it is a tree.
Although the structure of $G^t$ or $L(G)^t$ can be difficult to comprehend for general $G$, it is straightforward to characterise the extremal behaviour of $\chi'_t(G)$ or $\chi_t(G)$ if $G$ has no cycles.
For $t\ge 1$ and $d\ge 3$, define
\begin{align*}
\tau'_t(d) & := 
\begin{cases}
\frac{1}{d-2}(2(d-1)^{t/2+1}-d) & 2\mid t \\
\frac{1}{d-2}(d(d-1)^{(t+1)/2}-d) & 2\nmid t
\end{cases} \ \ \text{ and}
\\
\tau_t(d) & :=
\begin{cases}
\frac{1}{d-2}(d(d-1)^{t/2}-2) & 2\mid t\\
\frac{1}{d-2}(2(d-1)^{(t+1)/2}-2) & 2\nmid t
\end{cases}.
\end{align*}
\begin{proposition}\label{prop:tree}
Let $t\ge 1$ and $d\ge 3$.
There is a tree $G$ of maximum degree $d$ with $\tau'_t(d)$ edges ($\tau_t(d)$ vertices) such that $L(G)^t$ ($G^t$) is a clique.
If $G$ is a tree of maximum degree $d$, then $\chi'_t(G)\le \tau'_t(d)$ and $\chi_t(G)\le \tau_t(d)$.
\end{proposition}
\noindent
The examples in Proposition~\ref{prop:tree} are merely subgraphs of the infinite $d$-regular tree formed by including all vertices or edges within a suitable fixed distance of some root vertex or edge, depending on the parity of $t$.

The high waterline for us is when $G$ is a planar (multi)graph of maximum degree $d$.
In this case $\chi'_t(G)$ or $\chi_t(G)$ is still at most a constant factor greater than $\tau'_t(d)$ or $\tau_t(d)$, but it can well be a non-trivial factor greater. 
\begin{proposition}\label{prop:planar}
Let $t\ge 1$.
Let $\pi'_t(d)$ ($\pi_t(d)$) be the largest value of $\chi'_t(G)$ ($\chi_t(G)$) over the class of planar multigraphs of maximum degree $d$.
Then $\limsup_{d\to\infty} \pi'_t(d)/\tau'_t(d) \in[3/2,\infty)$ and $\limsup_{d\to\infty} \pi_t(d)/\tau_t(d)  \in[3/2,\infty)$.
\end{proposition}
\noindent
Bounds in Proposition~\ref{prop:planar} for distance vertex-colouring were established generally for all $t$ in~\cite{AgHa03} (upper bounds) and~\cite{FHS98} (lower bounds).
It is difficult to determine the precise values in Proposition~\ref{prop:planar}, especially for distance vertex-colouring: 
$\limsup_{d\to\infty} \pi_1(d)/\tau_1(d)=2$ is the Four Colour Theorem~\cite{ApHa77,AHK77}, while $\limsup_{d\to\infty} \pi_2(d)/\tau_2(d) = 3/2$ is the asymptotic confirmation of Wegner's Conjecture~\cite{HHMR08}.
Also $\limsup_{d\to\infty} \pi'_1(d)/\tau'_1(d) = 3/2$ is Shannon's Theorem~\cite{Sha49} and in Conjecture~\ref{conj:strongplanar} (see also Proposition~\ref{prop:strongcliqueplanar}) we posit $\limsup_{d\to\infty} \pi'_2(d)/\tau'_2(d) = 9/4$. For $t\ge3$, no precise value is known.

Our main result completely resolves the question displayed earlier and says that $\chi'_t(G)$ or $\chi_t(G)$ for planar graphs $G$ of maximum degree $d$ and large enough girth must be at most $\tau'_t(d)$ or $\tau_t(d)$, respectively, but only provided $t$ is of the correct parity. Informally, planar distance colouring becomes tree-like for high enough girth, in the right parities of distance.

\begin{theorem}\label{thm:main}
For odd (even) $t\ge 1$, there exists $g'_t$ ($g_t$) such that, 
provided $d$ is large enough, every planar graph $G$ of maximum degree $d$ and girth at least $g'_t$ ($g_t$) has $\chi'_t(G)\le \tau'_t(d)$ ($\chi_t(G)\le \tau_t(d)$). For the other parity of $t$, such a quantity does not exist. 
\end{theorem}

\noindent
By Proposition~\ref{prop:tree}, the bounds are sharp.
As mentioned, $g'_1$ is optimally at most $3$. It has also been shown that $g_2$ is optimally at most $7$~\cite{BGINT04}. For larger $t$, we have explicit bounds on $g'_t$ and $g_t$ (see Theorems~\ref{thm:girthedge},~\ref{thm:girthedgeimproved}, and~\ref{thm:girthvertex}), but we made little attempt to optimise them. On the contrary, we opted for clean proofs that apply in general. 

It remains unclear what is the situation for the ``other'' distance parities. Are there natural analogues to Gr\"otzsch's Theorem in those cases?

\begin{problem}
For even (odd) $t\ge 1$, what is the least $c'_t\ge1$ ($c_t\ge1$) such that there exists some fixed $g$ for which the following holds?
For every $\eps>0$, provided $d$ is large enough, every planar graph $G$ of maximum degree $d$ and girth at least $g$ has $\chi'_t(G)\le (c'_t+\eps) \tau'_t(d)$ ($\chi_t(G)\le (c_t+\eps) \tau_t(d)$).
\end{problem}

\noindent
By Proposition~\ref{prop:planar}, the constants $c'_t$ and $c_t$ are well-defined, but we do not know yet if  $c'_t < \limsup_{d\to\infty} \pi'_t(d)/\tau'_t(d)$ or  $c_t < \limsup_{d\to\infty} \pi_t(d)/\tau_t(d)$ in general.
Gr\"otzsch's Theorem says $c_1 \le 3/2 < \limsup_{d\to\infty} \pi_1(d)/\tau_1(d)$, and it also implies
that $c'_2 \le 3/2 < \limsup_{d\to\infty} \pi'_2(d)/\tau'_2(d)$~\cite{CMPR14}.

Moreover, there remains the possibility that, for even $t$, there is a function $g'_t(d)$ (so with a dependence upon $d$) such that every planar graph $G$ of maximum degree $d$ and girth at least $g'_t(d)$ has $\chi'_t(G)\le \tau'_t(d)$, provided $d$ is large enough. The existence of $g'_2(d)$ and it being at most linear in $d$ was proven in~\cite{BoIv13}.
For odd $t$ and distance-$t$ vertex-colouring the existence of an analogous $g_t(d)$ is impossible.
This follows from Proposition~\ref{prop:other}, where we also prove for even $t$ that $g'_t(d)$, if it exists, is at least of order $\tau'_t(d)$.

\subsubsection*{Structure of the paper}

In Section~\ref{sec:trees}, we describe the proof of Proposition~\ref{prop:tree} as a warm up. In Section~\ref{sec:planar} we prove Proposition~\ref{prop:planar} assuming the result of Agnarsson and Halld\'orsson. We also state a conjecture about the optimal value of the strong chromatic index of planar multigraphs and give some supporting evidence. In Section~\ref{sec:girth}, we prove Theorem~\ref{thm:main}, with more emphasis on distance edge-colouring. Our paper ends with the description of certificates for non-existence of a valid girth constraint in the ``other'' parities of $t$.

\subsubsection*{Notation and terminology}

Let $G$ be a graph. Then $V(G)$ and $E(G)$ are the vertex and edge set of $G$, respectively.
For a vertex $v \in V(G)$, we write $N(v)$ for the set of neighbours of $v$ in $G$ and $\deg(v)$ for the degree of $v$, i.e.~$\deg(v)=|N(v)|$.
If $\deg(v) = k$, we call $v$ a $k$-vertex.
If $G$ is a plane graph, i.e.~a graph with an embedding in the plane, we write $F(G)$ for the face set of $G$. For $f \in F(G)$, we write $\deg(f)$ for the number of edges counted in a closed walk along the boundary around $f$ (so edges visited twice are counted twice).

For two vertices $v_1,v_2\in V(G)$, the distance between $v_1$ and $v_2$ refers to the number of edges in a shortest path between $v_1$ and $v_2$.
For two edges $e_1,e_2\in V(G)$, the distance between $e_1$ and $e_2$ refers to the distance between their respective vertices in the line graph $L(G)$.
For a vertex $v\in V(G)$ and edge $e\in E(G)$, the distance between $v$ and $e$ is the number of edges in a shortest path between $v$ and one of the endpoints of $e$.

For an edge $e\in E(G)$ (a vertex $v\in V(G)$), we write $N_k(e)$ ($N_k(v)$) for the set of edges $\ne e$ (vertices $\ne v$) at distance at most $k$ in $L(G)$  (in $G$) from the edge $e$ (vertex $v$), and call it the distance-$k$ neighbourhood of $e$ ($v$).

We refer to a distance-$t$ edge-colouring (vertex-colouring) as a colouring $c: E(G)\to \mathbb{Z}^+$ ($c: V(G)\to \mathbb{Z}^+$) of the edges (vertices) so that every two distinct edges (vertices) of $G$ at distance at most $t$ have different colours.

\subsubsection*{A simple tree we repeatedly use}
Since we use it several times throughout the paper, for integers $k\ge 0$ and $d\ge 3$ let us define ${\mathcal T}_{k,d}$ as the internally $d$-regular tree of height $k$ that is rooted at the endpoint of a leaf. So this is a single vertex if $k=0$, a single edge if $k=1$, and a $d$-pointed star if $k=2$.
We also write
\begin{align*}
\iota_{k,d} := \sum_{i=0}^{k-1}(d-1)^i=\frac{1}{d-2}((d-1)^k-1).
\end{align*}
Note ${\mathcal T}_{k,d}$ has $\iota_{k,d}$ edges and the same number of non-root vertices.

\section{Trees}\label{sec:trees}

The result is not new, but we for completeness explain the simple proof of Proposition~\ref{prop:tree}. Since the arguments are very similar for distance edge- and vertex-colouring, we only treat distance edge-colouring.

\begin{proof}[Proof of Proposition~\ref{prop:tree} for $\chi'_t$]
We start with the constructions. If $t$ is odd, then let $G$ be the subgraph of the infinite $d$-regular tree defined by taking all vertices within distance $(t+1)/2$ of some fixed root vertex and letting $G$ be the induced subtree. Then $G$ is the same as $d$ copies of ${\mathcal T}_{(t+1)/2,d}$ all rooted at the same vertex. The number of edges in $G$ is (since $d\ge 3$)
\begin{align*}
d\cdot \iota_{(t+1)/2,d} = \tau'_t(d)
\end{align*}
and, via the root vertex, all edges are within distance $t$ in $L(G)$, i.e.~$L(G)^t$ is a clique and $\chi'_t(G) = \tau'_t(d)$.

If instead $t$ is even, then take all vertices within distance $t/2$ of some fixed root edge for $G$. Then $G$ is the same as the graph formed from a single edge by rooting $d-1$ copies of ${\mathcal T}_{t/2,d}$ to each of its endpoints. The number of edges in $G$ is (since $d\ge 3$)
\begin{align*}
1+2(d-1)\cdot \iota_{t/2,d} = \tau'_t(d).
\end{align*}
Via the root edge or its endpoints, all edges are within distance $t$ in $L(G)$.

Next for the upper bound, assume $G$ is some tree with maximum degree $d$.
If $t$ is odd, then order the edges of $G$ as follows. Nominate an arbitrary root vertex and perform a breadth-first search rooted at that vertex. Order the edges according to their first traversal time during this search. Colour the edges greedily, so each time choosing as a colour the least positive integer available under the distance constraint. By the properties of the breadth-first search ordering, at each step the number of previously coloured edges within distance $t$ (in $L(G)$) is strictly smaller than $\tau'_t(d)$. So the number of colours used does not exceed $\tau'_t(d)$.

If instead $t$ is even, then we order the edges of $G$ according to a breadth-first search tree rooted at some nominated edge. Similarly, the number of colours used during a greedy colouring procedure performed according to this ordering does not exceed $\tau'_t(d)$.
\end{proof}

\section{Planar (multi)graphs}\label{sec:planar}

In this section, we first prove Proposition~\ref{prop:planar} and then make a few specific remarks about the strong chromatic index of planar multigraphs.

Agnarsson and Halld\'orsson previously established the upper bounds in Proposition~\ref{prop:planar} for distance vertex-colouring by suitably bounding the arboricity of $G^t$ for any planar graph $G$ of maximum degree $d$.

\begin{theorem}[Agnarsson and Halld\'orsson~\cite{AgHa03}]\label{thm:AgHa}
Let $t\ge 1$.
There is a constant $K_t$ such that for all $d$ every planar graph $G$ of maximum degree $d$ has $\chi_t(G) \le K_t\cdot \tau_t(d)$. In other words, $\pi_t(d)/\tau_t(d) \le K_t$ for all $d$.
\end{theorem}

\noindent
For the corresponding distance edge-colouring upper bounds, we substitute Theorem~\ref{thm:AgHa} into an elegant argument due to Faudree {\em et al.}~\cite{FSGT90}. Faudree {\em et al.}~settled the case $\chi'_2$ for simple planar graphs with an analogous constant of $4$ (which is sharp).
It is worth mentioning that the distance edge-colouring bounds in  Proposition~\ref{prop:planar} considerably improve upon earlier bounds of Jendrol' and Skupie\'n~\cite{JeSk01}.

\begin{proof}[Proof of Proposition~\ref{prop:planar}]
As just mentioned, for the upper bounds it only remains to consider distance edge-colouring by Theorem~\ref{thm:AgHa}.
Let $G$ be a planar multigraph of maximum degree $d$. By Shannon's Theorem~\cite{Sha49}, $\chi'_1(G) \le 3d/2$. This not only implies that we may assume $t\ge 2$, but also it certifies that there is a proper edge-colouring of $G$ using colours from $\{1,\dots,\lfloor 3d/2\rfloor\}$.
Define for each colour $i$ the set $E_i$ as the set of edges of $G$ with colour $i$, and $H_i$ the simple graph which is formed from $G$ by contracting all edges in $E_i$ and then removing any loops or multiple edges. Since edge contraction preserves planarity, $H_i$ is planar for every $i$. Furthermore, the maximum degree of $H_i$ is at most $2d-2$. By Theorem~\ref{thm:AgHa}, $H_i$ admits a distance-$t$ vertex-colouring $c_i: V(H_i)\to \lfloor K_{t-1}\cdot \tau_{t-1}(2d-2)\rfloor$. Consider the partial edge-colouring of $G$ that assigns to each edge $e$ of $E_i$ the colour assigned by $c_i$ to the vertex in $H_i$ corresponding to the contracted $e$. Note that, by the definition of the $H_i$ and $c_i$, the edge-colouring of $G$ that combines all $\lfloor 3d/2\rfloor$ of these partial edge-colourings taken over disjoint colour sets has the correct distance property. Moreover, the number of colours used is
\begin{align*}
\lfloor 3d/2\rfloor \cdot \lfloor K_{t-1}\cdot \tau_{t-1}(2d-2)\rfloor \le (3\cdot 2^{\lfloor t/2\rfloor-1} \cdot K_{t-1}+\eps) \cdot \tau'_t(d)
\end{align*}
(for $d$ large enough), which implies the desired upper bound.

\smallskip

We next present the lower bound certificates. We give two types of construction. With respect to distance vertex-colouring, similar constructions with the same leading constants $3/2$ and $9/2$ were given by Fellows, Hell and Seyffarth~\cite{FHS98}. The constructions we give are simple and may furnish insight prior to our main result.
Without loss of generality assume $d$ is even.

Let ${\mathcal S}_d$ be the $d$-regular Shannon multigraph where each edge of the triangle has multiplicity $d/2$. This multigraph has $3d/2$ edges and $3$ vertices.
Note that ${\mathcal S}_d$ is planar and $L({\mathcal S}_d)$ is a clique.

Let ${\mathcal O}_d$ be the planar multigraph formed from the graph of the octahedron as follows. Let $\tau_1$ and $\tau_2$ be two triangles corresponding to opposite faces of the octahedron. Give every edge of $\tau_1$ multiplicity $d/2-1$. For each vertex of $\tau_2$ attach $d-4$ pendant edges. By construction, ${\mathcal O}_d$ has maximum degree $d$ and it has $9d/2-3$ edges and $3d-6$ vertices.
It can be routinely verified that $L({\mathcal O}_d)^2$ is a clique.

For each integer $k\ge 0$, we construct the planar graphs ${\mathcal S}_{k,d}$ and ${\mathcal O}_{k,d}$ as follows.
Root $d-2$ copies of ${\mathcal T}_{k,d}$ at the centre of each multiple edge (so identifying their roots), and root $d-1$ copies of ${\mathcal T}_{k,d}$ at each pendant vertex.
Note that ${\mathcal S}_{0,d}$ coincides with Wegner's construction~\cite{Weg77}.

Some arithmetic shows that ${\mathcal S}_{k,d}$ has $(3d/2)\cdot((d-1)^k+1)$
edges and $(3d/2)\cdot(d-1)^k+3$
vertices, and that ${\mathcal O}_{k,d}$ has $(9d/2-15)\cdot((d-1)^k-1)+6(d-1)$
edges and $(9d/2-15)\cdot((d-1)^k-1)+9d/2-9$
vertices.

By construction, the following assertions can be deduced directly from the distance properties of ${\mathcal S}_d$, ${\mathcal O}_d$ and ${\mathcal T}_{k,d}$.
Both ${{\mathcal S}_{k,d}}^{2k+2}$ and ${{\mathcal O}_{k,d}}^{2k+3}$ are cliques.
If $k\ge 1$, then $L({\mathcal S}_{k,d})^{2k+1}$ and $L({\mathcal O}_{k,d})^{2k+2}$ are cliques.

By the above, we have that
\[
\limsup_{d\to\infty} \frac{\pi'_t(d)}{\tau'_t(d)} \ge 
\begin{cases}
9/4 & 2\mid t \\
3/2 & 2\nmid t
\end{cases}
\ \ \text{ and }\ \
\limsup_{d\to\infty} \frac{\pi_t(d)}{\tau_t(d)} \ge
\begin{cases}
3/2 & 2\mid t\\
9/4 & 2\nmid t
\end{cases}. \qedhere
\]
\end{proof}

The distance-$t$ edge-colouring cases $t=2$ and $t=3$ are tantalising. By the Four Colour Theorem~\cite{ApHa77,AHK77} and the asymptotic form of Wegner's conjecture~\cite{HHMR08}, we obtain from the above argument that $\pi'_2(d)/\tau'_2(d) \le 3$ and $\limsup_{d\to\infty} \pi'_3(d)/\tau'_3(d) \le 9/2$, respectively. Comparing with the examples ${\mathcal O}_d$ and ${\mathcal S}_{1,d}$ corresponding to lower bounds of $9/4$ and $3/2$, respectively, it would be very interesting to sharpen the gaps.

With respect to the strong chromatic index $\chi'_2$ in particular, such lower bounds were not previously observed as far as we know. Note that the behaviour here contrasts with the general (not necessarily planar) case, where we believe that allowing edges of multiplicity has no effect upon the extremal behaviour of $\chi'_2$~\cite{JKP17+}. Within the class of planar multigraphs, we posit that the examples ${\mathcal O}_d$ are asymptotically extremal for $\chi'_2$.

\begin{conjecture}\label{conj:strongplanar}
Fix $\eps>0$. For every planar multigraph $G$ of maximum degree $d$, the strong chromatic index $\chi'_2(G)$ of $G$ satisfies $\chi'_2(G) \le (9/2+\eps)d$, provided $d$ is sufficiently large.
\end{conjecture}

\noindent
This problem concerns sufficiently large $d$. We point out though that recently the smallest non-trivial case $d=3$ was resolved, with no difference in extremal behaviour in comparison with simple planar graphs~\cite{KLRSWY16}.

In support of Conjecture~\ref{conj:strongplanar}, we show that, within the class of planar multigraphs (and, in fact, within the slightly more general class of multigraphs whose underlying simple graph is $K_5$-minor-free), the examples ${\mathcal O}_d$ are nearly extremal for a parameter (akin to diameter) that is closely related to but weaker than $\chi'_2$.

\begin{proposition}\label{prop:strongcliqueplanar}
Every planar multigraph $G$ of maximum degree $d$ such that $L(G)^2$ is a clique satisfies $|E(G)|\le 9d/2$.
\end{proposition}

\begin{proof}
Let $G$ be a planar multigraph of maximum degree $d$ such that $L(G)^2$ is a clique.
Let $\hat{G}$ denote the underlying planar simple graph of $G$ and note that $L(\hat{G})^2$ is also a clique.
For any two edges in a matching there must be at least one edge of $\hat{G}$ between their endpoints. Letting $k$ be the number of edges in a largest matching in $\hat{G}$, then by contraction there must be a $k$-clique-minor, and so $k\le 4$ by the planarity of $\hat{G}$.

By the Tutte--Berge Formula (cf.~e.g.~\cite[Cor.~16.12]{BoMu08}), if $\odd(H)$ denotes the number of connected components in $H$ with an odd number of vertices,
\begin{align}\label{eqn:bergetutte}
k =\frac12\min_{U \subseteq V(\hat{G})}(|U|-\odd(\hat{G}-U)+|V(\hat{G})|).
\end{align}

Fix some $U \subseteq V(\hat{G})=V(G)$ achieving the minimum in~\eqref{eqn:bergetutte}.
Since $L(\hat{G})^2$ is a clique, $\hat{G}-U$ has at most one component $A$ with any edges. 
If $\hat{G}-U$ has no such component, then $|U|=k \le 4$.
Otherwise, if $|V(A)|$ is even, then the odd components are all singletons and~\eqref{eqn:bergetutte} becomes
\begin{align*}
k
&=\frac12(|U|-\odd(\hat{G}-U)+(|U|+\odd(\hat{G}-U)+|V(A)|))
=|U|+\frac12|V(A)|.
\end{align*}
If instead $|V(A)|$ is odd, then there is no even component and the other odd components are all singletons, so~\eqref{eqn:bergetutte} becomes
\begin{align*}
k
=|U|+\frac12(|V(A)|-1).
\end{align*}
In either case, $|U|+|V(A)|/2 \le k+1/2 \le 9/2$.

Since $G$ has maximum degree $d$, at most $|U|\cdot d$ edges in $G$ are incident to the vertices of $U$. There are at most $(|V(A)|/2)\cdot d$ remaining edges in $G$ (when $\hat{G}-U$ has some non-singleton component $A$), since those edges span $V(A)$. Then the above estimate implies there are at most $9d/2$ edges.
\end{proof}

\section{High-girth planar graphs}\label{sec:girth}

In this section, we prove our main result. We give the arguments in full for distance edge-colouring and later sketch how it can be adapted to distance vertex-colouring. Note that the argument that is used also applies for $t=1$, but Vizing's result~\cite{Viz65} is much stronger in that case.

\begin{theorem}\label{thm:girthedge}
Fix $t\ge 3$ odd. If $G$ is a planar graph of maximum degree $d\ge 4$ and girth at least $6(t^2+2t-1)$, then $\chi'_t(G)\le \tau'_t(d)$.
\end{theorem}

\begin{proof}
For a contradiction, assume that $G$ is a counterexample that minimises $|V(G)|+|E(G)|$. More specifically, let $G$ be a plane graph (i.e.~with an embedding in the plane) of maximum degree $d\ge 4$ and girth at least $6(t^2+2t-1)$ for which $\chi'_t(G)> \tau'_t(d)$, but $\chi'_t(G') \le \tau'_t(d)$ for any proper subgraph $G'$ of $G$. Since $G$ is minimal, it must be connected.

For $k\ge 2$, we say that a $k$-vertex $v$ of $G$ is {\em adjacent to a short tree}, if there exists a vertex $u\in N(v)$ such that when we remove the edge $uv$, $u$ is disconnected from $v$ and the maximal subgraph containing $u$ is a tree. We call this subgraph, including the edge $uv$, a {\em short tree}, and call $v$ its {\em root} and $u$ its {\em stem}. 
Before continuing, let us justify our use of the adjective `short'.

\begin{claim}\label{clm:shorttree}
Any short tree has height at most $(t-1)/2$.
\end{claim}

\begin{claimproof}
Suppose that $G$ contains a vertex $v$ adjacent to a short tree with height $s\ge (t+1)/2$. Let $e$ be a (leaf) edge at distance $s-1$ from $v$. 
By a breadth-first count rooted at a vertex at distance $(t-1)/2$ along the path between $e$ and $v$, we see that 
\begin{align*}
|N_t(e)| 
& \le d+d(d-1)+d(d-1)^2+\dots+d(d-1)^{(t-1)/2} - 1
 = \tau'_t(d) - 1.
\end{align*}
By minimality, $\chi'_t(G-\{e\}) \le \tau'_t(d)$ and so there is a corresponding distance-$t$ colouring of $G-\{e\}$. To extend this to $G$, for the edge $e$ we only need to choose a colour that does not appear on $N_t(e)$, which is possible by the above estimate. This is a contradiction.
\end{claimproof}

\noindent
Note that there are at most $\iota_{s,d}$ edges in a short tree of height $s (\le (t-1)/2)$.

For $k\ge 2$ we say that a $k$-vertex is {\em weak} if it is adjacent to exactly $k-2$ short trees.
We now give two structural properties of $G$.

\begin{claim}\label{clm:girthedge,1}
Each $k$-vertex in $G$ is adjacent to at most $k-1$ short trees.
\end{claim}
\begin{claimproof}
If $G$ contains a $k$-vertex that is adjacent to $k$ short trees, then $G$ is a tree. By Proposition~\ref{prop:tree}, $\chi'_t(G)\le \tau'_t(d)$, a contradiction.
\end{claimproof}

\begin{claim}\label{clm:girthedge,2}
Any path in $G$ of weak vertices contains at most $t^2+2t-2$ vertices.
\end{claim}

\begin{claimproof}
Suppose to the contrary that $G$ contains a path $v_1\cdots v_{t^2+2t-1}$ of $t^2+2t-1$ weak vertices.
Then one of the following statements must hold: 
\begin{enumerate}
\item\label{case1} $G$ contains a path of $3t-1$ weak vertices all of degree $2$; or 
\item\label{case2} $G$ contains a path $w_1\cdots w_{2t+1}$ of $2t+1$ weak vertices where $w_{t+1}$ is adjacent to a short tree of some height $s\le (t-1)/2$, and any short tree adjacent to a vertex $w_i$, $i\in\{1,\dots, 2t+1\}$, is of height at most $s$.
\end{enumerate}
To see this, suppose that~\ref{case2} is false.
By Claim~\ref{clm:shorttree}, the only possible heights available for a short tree are in $\{1,\dots,(t-1)/2\}$. By the assumption, no short tree of height $(t-1)/2$ is adjacent to any of the vertices $v_{t+1},\dots,v_{t^2+t-1}$. But this then implies, again by the assumption, that no short tree of height $(t-3)/2$ is adjacent to any of the vertices $v_{2t+1},\dots,v_{t^2-1}$. By an induction, it thus follows that no short tree of height $(t+1)/2-k$ is adjacent to any of the vertices $v_{kt+1},\dots,v_{t^2+(2-k)t-1}$, for $k\in\{1,\dots,(t-1)/2\}$. For $k=(t-1)/2$, this implies that the path $v_{t^2/2-t/2+1},\dots,v_{t^2/2+5t/2-1}$ is not adjacent to any (non-trivial) short tree at all, implying that~\ref{case1} is true.

To prove the claim, we examine cases~\ref{case1} and~\ref{case2} separately.

First suppose~\ref{case1} holds, and let $u_1\cdots u_{3t-1}$ be a path of $3t-1$ weak vertices all of degree $2$. Let $H=G-\{u_{t+1}, u_{t+2}, \dots, u_{2t-1}\}$, so it is $G$ with the middle $t-1$ vertices of the path removed. By minimality, $\chi'_t(H)\le \tau'_t(d)$ and there exists a corresponding distance-$t$ edge-colouring of $H$. We can extend this to a distance-$t$ colouring of $G$, by colouring the uncoloured edges greedily in any order. This uses no new colours because $N_t(u_iu_{i+1})= 2t\le \tau'_t(d)-1$ (using the assumption that $d\ge 4$) for all $i\in \{t+1,\dots, 2t-2\}$. So we have reached a contradiction.

Next suppose~\ref{case2} holds. Let $e$ be an edge (in the short tree adjacent to $w_{t+1}$ of height $s$) that is at distance $s-1$ from $w_{t+1}$, and let $H=G-\{e\}$. By minimality, $\chi'_t(H)\le \tau'_t(d)$ and there exists a corresponding distance-$t$ edge-colouring of $H$. If we succeed in showing that $N_t(e)\le \tau'_t(d)-1$, then we can extend to a distance-$t$ edge-colouring of all of $G$, a contradiction, thereby proving Claim~\ref{clm:girthedge,2}.

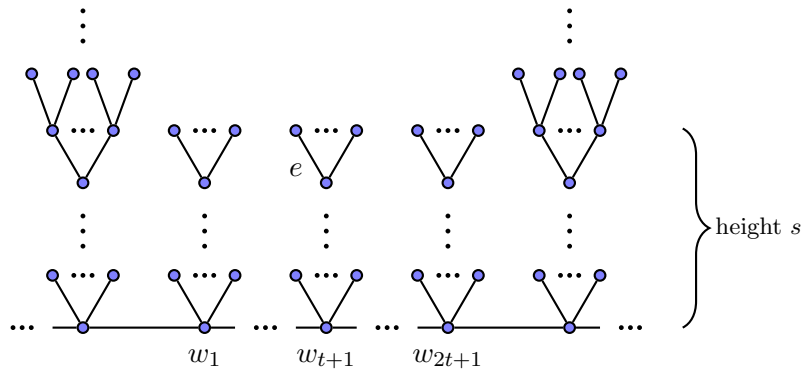
\begin{figure}
\begin{center}
	\hspace*{2cm}{\begin{tikzpicture}[thick,scale=0.8]
	\foreach \x in {-20,-10,0,10,20}{
		\node [node] (\x) at (\x/5,0) {};
	};
	\draw (-4.5,0) -- (-20) -- (-10) -- (-1.5,0);
	\draw (-0.5,0) -- (0) -- (0.5,0);
	\draw (1.5,0) -- (10) -- (20) -- (4.5,0);
	\foreach \y in {-1,1} {
		\foreach \z in {-1,0,1} {
			\node [point] at (\y + 0.15*\z,0) {};
		}
	};
	\foreach \y in {-1,1} {
		\foreach \z in {-1,0,1} {
			\node [point] at (5*\y + 0.15*\z,0) {};
		}
	};	
	\foreach \x in {-20,-10,0,10,20}{
		\foreach \y in {-1,1}{
			\node [node] (\x+\y) at ($ (\x) + (90+30*\y:1) $) {};
			\draw (\x) -- (\x+\y);
		}
	};
	\foreach \y in {-20,-10,0,10,20}{
		\foreach \z in {-1,0,1} {
			\node [point] at ($ (\y) + (0, 1.6 + 0.25*\z) $) {};
		}
	};	
	\foreach \y in {-20,-10,0,10,20}{
		\foreach \z in {-1,0,1} {
			\node [point] at ($ (\y+1) + (0.5+0.15*\z ,0) $) {};
		}
	};
	\foreach \x in {-20,-10,0,10,20}{
		\node [node] (\x+) at (\x/5,2.4) {};
	}		
	\foreach \x in {-20,-10,0,10,20}{
		\foreach \y in {-1,1}{
			\node [node] (\x++\y) at ($ (\x+) + (90+30*\y:1) $) {};
			\draw (\x+) -- (\x++\y);
		}
	};
	\foreach \y in {-20,-10,0,10,20}{
		\foreach \z in {-1,0,1} {
			\node [point] at ($ (\y++1) + (0.5+0.15*\z ,0) $) {};	
		}
	};
	\foreach \x in {-20,20}{
		\foreach \y in {-1,1}{
			\foreach \z in {-1,1}{
				\node [node] (\x+++\z) at ($ (\x++\y) + (90+20*\z:1) $) {};
				\draw (\x++\y) -- (\x+++\z);
			}
		}
	};
	\foreach \y in {-20,20}{
		\foreach \z in {-1,0,1} {
			\node [point] at ($ (\y) + (0, 5 + 0.25*\z) $) {};	
		}
	};		
			
	\draw (-10) node [label={[label distance=0cm]-90:$w_1$}] {};
	\draw (0) node [label={[label distance=0cm]-90:$w_{t+1}$}] {};
	\draw (10) node [label={[label distance=0cm]-90:$w_{2t+1}$}] {};
	\draw(0++1) node [label={[label distance=0.15cm]-90:$e$}] {};
	\draw [decorate,decoration={brace,amplitude=10pt, mirror},xshift=-4pt,yshift=0pt]
	(6,0) -- (6,3.3) node [black,midway,xshift=1cm] 
	{\footnotesize height $s$};

	\end{tikzpicture}}
\end{center}
\caption{A depiction of $N_t(e)$ under~\ref{case2} in Claim~\ref{clm:girthedge,2}.} \label{fig:2.4}
\end{figure}

Since the distance-$t$ neighbourhood of $e$ consists only of short trees of height at most $s$ (see Figure~\ref{fig:2.4}), we obtain
\begin{align*}
|N_t(e)|
& \le (d-2)\cdot\iota_{s,d}-1+2(t-2s+1)((d-2)\cdot\iota_{s,d}+1)+2\iota_{s,d}\\
& = (2t-4s+3)(d-1)^s-2+2\frac{(d-1)^s-1}{d-2}.
\end{align*}
We are done if this final expression is shown to be at most $\tau'_t(d)-1$.
After an arithmetic rearrangement, our aim then is to show the following inequality:
\begin{align}\label{eqn:2.2}
(2t-4s+3)(d-2)+2\le d(d-1)^{(t+1)/2-s}.
\end{align}
To show this, first consider the boundary case $s=(t-1)/2$:~\eqref{eqn:2.2} becomes
\begin{align*}
5d-8\le d(d-1),
\end{align*}
which holds true if $d\ge 4$. If we decrease the value of $s$ from $(t-1)/2$ to $(t-1)/2-k$ for some $1\le k\le ((t-1)/2)-1$, then in~\eqref{eqn:2.2} the left-hand side increases additively by the value $4k(d-2)$, while the right-hand side increases by a multiplicative factor of $(d-1)^k$. Initially the right-hand side was at least $12$ (because $d\ge 4$), so~\eqref{eqn:2.2} remains valid as we decrease $s$ from $(t-1)/2$ to $1$. This confirms~\eqref{eqn:2.2} and concludes the proof.
\end{claimproof}

The two structural properties established in Claims~\ref{clm:girthedge,1} and~\ref{clm:girthedge,2} are enough to complete the proof.
From $G$ we construct another plane graph as follows. Delete every vertex having a unique incident face. (After this, every vertex has degree at least $2$ by Claim~\ref{clm:girthedge,1}.) Then replace every maximal path of $2$-vertices by an edge connecting its endpoints, and call the resulting plane graph $G'$. By the girth assumption for $G$ and Claim~\ref{clm:girthedge,2}, $G'$ is a simple graph of minimum degree $3$ and girth at least $6$. Thus $2|E(G')| \ge 3|V(G')|$ and $2|E(G')| \ge 6|F(G')|$. By Euler's Formula,
\[
2= |V(G')|-|E(G')|+|F(G')| \le \frac23|E(G')|-|E(G')|+\frac13|E(G')|=0,
\]
a contradiction.
\end{proof}

Theorem~\ref{thm:girthedge} has an assumption that the maximum degree $d$ is at least $4$ (used for instance in the proof of~\eqref{eqn:2.2}). If $G$ is a planar graph with maximum degree $2$, then every component of $G$ is a path or a cycle. For paths $G$, $\chi'_t(G)\le t+1$, and this bound is met by a path of length $t+1$. On the other hand, $\chi'_t(G) > t+1$ for cycles $G$ of arbitrarily large length $\ell$, provided $\ell \equiv 1 \pmod{k+1}$.
We do not know yet if there is some girth $g'_{t,3}$ such that, if $G$ is a planar graph of maximum degree $3$ and girth at least $g'_{t,3}$, then $\chi'_t(G) \le \tau'_t(3) = 3\cdot 2^{(t+1)/2}-3$.

\subsection{Tradeoff between girth and degree}\label{sub:girthedgeimproved}

In this subsection, we briefly sketch how the girth condition in Theorem~\ref{thm:girthedge} can be reduced by a factor of $t$ at the expense of a lower bound condition of order $t$ on the maximum degree.

\begin{theorem}\label{thm:girthedgeimproved}
Fix $t\ge 3$ odd. If $G$ is a planar graph of maximum degree $d\ge t+2$ and girth at least $30t-6$, then $\chi'_t(G)\le \tau'_t(d)$.
\end{theorem}

\begin{proof}[Proof sketch]
In the proof of Claim~\ref{clm:girthedge,2}, we showed that one of two possibilities occur,~\ref{case1} and~\ref{case2}. In~\ref{case2}, we may drop the requirement that all of the short trees have height at most $s$ as follows:
\begin{itemize}
\item[($ii$')] $G$ contains a path $w_1\cdots w_{2t+1}$ of $2t+1$ weak vertices where $w_{t+1}$ is adjacent to a short tree.
\end{itemize}
To prove that~\ref{case1} or~($ii$') must hold, it suffices to start with a path of $5t-1$ weak vertices.
However, to complete the proof of this revised version of Claim~\ref{clm:girthedge,2}, it follows from the same arguments that the following must be true in order to establish the corresponding version of~\eqref{eqn:2.2}:
\begin{align*}
(t+3)(d-2)+2\le d(d-1).
\end{align*}
This holds if $d\ge t+2$.

The reducibility of a path of $5t-1$ weak vertices translates to a girth requirement of $6(5t-1)=30t-6$.
\end{proof}

We naturally wonder if this tradeoff can be pushed even further, possibly to a girth requirement independent of $t$, but we leave this to future study.

\subsection{Distance vertex-colouring}\label{sub:girthvertex}

In this subsection, we describe how to adapt the proof of Theorem~\ref{thm:girthedge} to distance-$t$ vertex-colouring.

\begin{theorem}\label{thm:girthvertex}
Fix $t\ge 2$ even. If $G$ is a planar graph of maximum degree $d\ge 4$ and girth at least $6(t^2+t-2)$, then $\chi_t(G)\le \tau_t(d)$.
\end{theorem}

\begin{proof}[Proof sketch]
Assume that $G$ is a counterexample that minimises $|V(G)|$. More specifically, let $G$ be a plane graph of maximum degree $d\ge 4$ and girth at least $6(t^2+t-2)$ for which $\chi_t(G)> \tau_t(d)$, but $\chi_t(G') \le \tau_t(d)$ for any proper induced subgraph $G'$ of $G$. Since $G$ is minimal, it must be connected.
We use the same notation as in the proof of Theorem~\ref{thm:girthedge}.

\begin{claim}\label{clm:shorttree,vertex}
Any short tree has height at most $t/2-1$.
\end{claim}

\begin{claimproof}
Suppose that $G$ contains a vertex $v$ adjacent to a short tree with height $s\ge t/2$. Let $u$ be a leaf vertex (so a vertex that lies at distance $s$ from $v$). Then 
\begin{align*}
|N_t(u)| 
& \le 1+d+d(d-1)+d(d-1)^2+\dots+d(d-1)^{t/2-1} - 1\\
& = \tau_t(d) - 1.
\end{align*}
By minimality, $\chi_t(G-\{u\}) \le \tau_t(d)$ and so there is a corresponding distance-$t$ colouring of $G-\{u\}$. To extend this to $G$, for the edge $u$ we only need to choose a colour that does not appear on $N_t(u)$, which is possible by the above estimate. This is a contradiction.
\end{claimproof}

\begin{claim}\label{clm:girthvertex,1}
Each $k$-vertex in $G$ is adjacent to at most $k-1$ short trees.
\end{claim}
\begin{claimproof}
If $G$ contains a $k$-vertex that is adjacent to $k$ short trees, then $G$ is a tree. By Proposition~\ref{prop:tree}, $\chi_t(G)\le \tau_t(d)$, a contradiction.
\end{claimproof}

\begin{claim}\label{clm:girthvertex,2}
Any path in $G$ of weak vertices contains at most $t^2+t-3$ vertices.
\end{claim}

\begin{claimproof}
Suppose to the contrary that $G$ contains a path $v_1\cdots v_{t^2+t-2}$ of $t^2+t-2$ weak vertices.
Then one of the following statements must hold: 
\begin{enumerate}
\item\label{case1,vertex} $G$ contains a path of $3t-2$ weak vertices all of degree $2$; or 
\item\label{case2,vertex} $G$ contains a path $w_1\cdots w_{2t+1}$ of $2t+1$ weak vertices where $w_{t+1}$ is adjacent to a short tree of height $s$, and any short tree adjacent to a vertex $w_i$, $i\in\{1,\dots, 2t+1\}$, is of height at most $s$.
\end{enumerate}
To see this, suppose that~\ref{case2,vertex} is false.
By Claim~\ref{clm:shorttree,vertex}, the only possible heights available for a short tree are in $\{1,\dots,t/2-1\}$. By the assumption, no short tree of height $t/2-1$ is adjacent to any of the vertices $v_{t+1},\dots,v_{t^2-2}$. But this then implies, again by the assumption, that no short tree of height $t/2-2$ is adjacent to any of the vertices $v_{2t+1},\dots,v_{t^2-t-2}$. By an induction, it thus follows that no short tree of height $t/2-k$ is adjacent to any of the vertices $v_{kt+1},\dots,v_{t^2+(1-k)t-2}$, for $k\in\{1,\dots,t/2-1\}$. For $k=t/2-1$, this implies that the path $v_{t^2/2-t+1},\dots,v_{t^2/2+2t-2}$ is not adjacent to any (non-trivial) short tree at all, implying that~\ref{case1,vertex} is true.

The rest of the proof is nearly identical to the proof of Claim~\ref{clm:girthedge,2}.
\end{claimproof}
The rest is the same as in the proof of Theorem~\ref{thm:girthedge}.
\end{proof}

\subsection{The ``other'' cases}\label{sub:other}

Here we give some examples certifying that it is impossible to extend Theorems~\ref{thm:girthedge} and~\ref{thm:girthvertex} to the other parities of $t$.

\begin{proposition}\label{prop:other}
For even $t\ge 2$, there exists a planar graph $G$ of maximum degree $d$ and girth at least $(\tau'_t(d)-1)/2$ such that $\chi'_t(G) > \tau'_t(d)$, provided $d$ is large enough.
For odd $t\ge 1$ and $d\ge 3$, there exists a planar graph $G$ of maximum degree $d$ and arbitrarily large girth such that $\chi_t(G) > \tau_t(d)$.
\end{proposition}

\begin{proof}
For distance-$t$ edge-colouring, fix $t\ge2$ even and $d$ large enough. The construction of $G$ is as follows. Take an odd cycle of length $\ell\in\{(\tau'_t(d)-1)/2,(\tau'_t(d)+1)/2\}$ and to each vertex root $d-2$ copies of the tree ${\mathcal T}_{t/2,d}$. This is a planar graph of maximum degree $d$ and girth at least $(\tau'_t(d)-1)/2$.

Write the cycle as $u_1\cdots u_\ell u_1$. Suppose for a contradiction that there is some distance-$t$ edge-colouring $c': E(G) \to \{1,\dots,\tau'_t(d)\}$.
Let us write $N'_k(v)$ for the set of edges at distance at most $k$ from the vertex $v$.
For each $i\in \{1,\dots,\ell\}$, let $X'_i = \cup\{N'_j(u_i)\setminus N'_{j-1}(u_i) :j\le t/2-1\text{ even}\}$.
Note that  $\ell > t$ for $d$ large enough and so $X'_i \cap X'_{i+1} = \{u_iu_{i+1}\}$ and $|X'_i|+|X'_{i+1}|-1=\tau'_t(d)$ for all $i$.
Since $c'$ is a distance-$t$ colouring and $X'_i\cup X'_{i+1}$ induces a clique in $L(G)^t$, it follows for every $i$ that the colour sets $c'(X'_i)$ and $c'(X'_{i+2})$ have a symmetric difference of size at most $2$.
If we apply this fact in sequence, $i=1,3,5,\dots,\ell-2$, then since $\ell \le (\tau'_t(d)+1)/2$ we derive that $c'(X'_\ell)$ and $c'(X'_1)$ have a colour in common, a contradiction.

\smallskip

For distance-$t$ vertex-colouring, fix $t\ge1$ odd and $d\ge 3$. The construction of $G$ is similar. Take an odd cycle of length $\ell$ at least $t+2$ and to each vertex root $d-2$ copies of the tree ${\mathcal T}_{(t-1)/2,d}$. This is a planar graph of maximum degree $d$ and girth $\ell$ (which can be taken arbitrarily large).

Write the cycle as $u_1\cdots u_\ell u_1$. Suppose for a contradiction that there is some distance-$t$ vertex-colouring $c: V(G) \to \{1,\dots,\tau_t(d)\}$.
For each $i\in \{1,\dots,\ell\}$, let $X_i = \cup\{N_j(u_i)\setminus N_{j-1}(u_i) :j\le (t-1)/2\text{ odd}\}$.
Note since $\ell>t$ that $X_i \cap X_{i+1} = \emptyset$ and $|X_i|+|X_{i+1}|=\tau_t(d)$ for all $i$.
Since $c$ is a distance-$t$ colouring and $X_i\cup X_{i+1}$ induces a clique in $G^t$, it follows for every $i$ that the colour sets $c(X_i)$ and $c(X_{i+2})$ are equal, implying in particular that $c(X_\ell)=c(X_1)$, a contradiction.
\end{proof}

\noindent
The examples of Proposition~\ref{prop:other} all simply build upon an odd cycle. It is conceivable that some finite girth constraint always ensures that the distance colouring problems in {\em bipartite} planar graphs resembles that of trees.

\subsubsection*{Acknowledgements}

The first author would like to thank Magn\'us Halld\'orsson for helpful discussions at a very early stage of this work.
We are grateful to an anonymous referee for helpful comments and suggestions.

\bibliographystyle{abbrv}
\bibliography{distplanar}

\end{document}